\documentclass[11pt]{amsart}
\usepackage{amsfonts}
\usepackage{amssymb}
\usepackage{hyperref}
\usepackage{graphics}
\usepackage{hyperref}
\usepackage{amsmath}
\hypersetup{
colorlinks=true,
linkcolor=red,
citecolor=cyan,
}

\newtheorem{thmx}{Theorem}[section]
\newtheorem{cor}[thmx]{Corollary}

\renewcommand{\thethmx}{\Alph{thmx}}
\newtheorem*{conj*}{Denjoy's Conjecture}

\theoremstyle{definition}

\newtheorem{exa}{Example}[section]
\numberwithin{equation}{section}

  \title[A problem of Heittokangas-Ishizaki-Tohge-Wen]{A problem of Heittokangas-Ishizaki-Tohge-Wen concerning a certain differential-difference equation }

    \author{Xuxu Xiang}
     \address{Xuxu Xiang\newline School of Mathematical Sciences, Guizhou Normal University, Guiyang, 550025, P.R. China. }
     \email{1245410002@qq.com}
  
     \author{Jianren Long*}
     \address{Jianren Long \newline School of Mathematical Sciences, Guizhou Normal University, Guiyang, 550025, P.R. China. }
     \email{longjianren2004@163.com}
     \thanks{}

   \keywords{Nevanlinna theory; entire solution; exponential polynomial; differential-difference equation\\
   *Corresponding author.}
     \subjclass{30D35, 39B32}

\begin{document}
\begin{abstract}
All the finite order  entire solutions  of 
\begin{equation*}
f^n(z)+q(z)e^{Q(z)}f^{(k)}(z+c)=P(z)  
\end{equation*}
are given, where $ q(z) $, $ Q(z), P(z) $ are polynomials, $ k $ and $ n \geq 2 $ are integers, and $ c \in \mathbb{C} \setminus \{0\} $.
This solves an open problem of Heittokangas-Ishizaki-Tohge-Wen [Bull. Lond. Math. Soc. 55, 1-77 (2023)] and complements the results of  Wen-Heittokangas-Laine [Acta Math. Sinica 28, 1295-1306 (2012)] and Liu [Mediterr. J. Math. 13, 3015–3027 (2016)]. 

\end{abstract}

\maketitle

%% The correct journal style for \specialsection is all uppercase; a known bug
%% in amsart.cls prevents this, so input must be uppercase until it is fixed.
%\specialsection*{This is a Special Section Head}
%\specialsection*{THIS IS A SPECIAL SECTION HEAD}
%This is an example of a special section head%
%%%%%%%%%%%%%%%%%%%%%%%%%%%%%%%%%%%%%%%%%%%%%%%%%%%%%%%%%%%%%%%%%%%%%%%%
%\footnote{Here is an example of a footnote. Notice that this footnote
%text is running on so that it can stand as an example of how a footnote
%with separate paragraphs should be written.
%\par
%And here is the beginning of the second paragraph.}%
%%%%%%%%%%%%%%%%%%%%%%%%%%%%%%%%%%%%%%%%%%%%%%%%%%%%%%%%%%%%%%%%%%%%%%%%
.

\section{Introduction and main results}\label{sec1}

~~~~
Let $f$ be a meromorphic function in the complex plane $\mathbb{C}$. Assume that the reader is familiar with the standard notation and basic results of Nevanlinna theory, such as the proximity function $m(r,f),$ the counting function $N(r,f)$,  the
characteristic function $~T(r,f)=m(r,f)+N(r,f)$, see \cite{hayman} for more details. A meromorphic
function $g$ is said to be a small function of $f$ if $T(r,g)=S(r,f)$, where $S(r,f)$  denotes any quantity
that satisfies $S(r,f)= o(T(r, f))$ as $r$ tends to infinity, outside of a possible exceptional set of finite linear measure. $\rho(f)$   denotes the order of $f$ and define $\lambda(f)$ as the exponent of convergence of the zeros sequences of $f$.  

%

 %In 1927,  Montel\cite{Montel} studied Fermat type functional equation\begin{align}\label{0.1.1}	f^n(z)+g^m(z)=1,\end{align}for the first time by drawing an analogy with Fermat's number-theoretic equations,where  $f,g$ are meromorphic functions, $m,n$ are positive integers. Montel proved  \eqref{0.1.1} has no nonconstant entire solutions under the conditions $m=n>2$. Later,  the existence of meromorphic solutions to equation \eqref{0.1.1} is discussed by Gross\cite{} and Baker\cite{B}.

In the last century, in order to study the zeros of meromorphic functions and their derivatives, Hayman \cite[p. 69]{hayman} provided a generalization of the Tumura-Clunie theorem \cite{clu,Tu}. 

\setcounter{thmx}{0}

\renewcommand{\thethmx}
{\arabic{section}.\arabic{thmx}}

\begin{thmx}\cite[Theorem 3.9]{hayman}
\label{thA}
 Let  meromorphic functions $f$ and $g$  satisfy the nonlinear differential equation
 \begin{align}
 \label{tc}
 f^n(z)+P_d(z,f)=g(z), 
 \end{align}
 where $P_d(z,f)$ is a differential polynomial in $f$ of degree $d\leq{n-1}$ with  small functions coefficients. If $N(r,f)+N(r,\frac{1}{g})=S(r,f)$, then  $g(z)=(f(z)+a(z))^n$, where
$a$ is a small meromorphic function of $f$.
\end{thmx}

There are  many   extensions and applications on the Tumura-Clunie theorem; see e.g. \cite{Mes,Toda,Yi}.   In 2001, Yang \cite{yang2001} investigated transcendental entire functions \( f \) of finite order that satisfy the equation
\begin{equation*}
L(f) - p(z)f^n = h(z),
\end{equation*}
where \( L(f) \) represents a linear differential polynomial in \( f \) with polynomial coefficients, \( p(z) \) is a non-zero polynomial, \( h(z) \) is an entire function, and \( n > 3 \). Specifically, Yang demonstrated that such a function \( f \) must be unique, except in the case where \( L(f) \) is identically zero. Heittokangas-Korhonen-Laine \cite{H2002}  complemented  Yang's results when $h$, $p$, and the coefficients of \( L(f) \) are meromorphic. 
 
 With the establishment and development of the difference Nevanlinna theory (see \cite{chiang,halburd2006,halburd2014} for more details), the study of the difference version of \eqref{tc} has become a very interesting topic. 
Yang-Laine \cite{yang2010} proved that when \( n > 3 \), $L(z,f)$ is a linear differential-difference polynomial in $f$ with small meromorphic coefficients and \( h \) is of finite order, equation 
 \begin{equation*}
f^n-L(z,f) = h(z)
\end{equation*}
 has at most one finite order  transcendental entire  solution, and \( f \) has the same order as \( h \). In particular,  They shown   that the equation
\begin{equation}
\label{1.2}
f(z)^2 + q(z)f(z + 1) = p(z),
\end{equation}
where \( p(z), q(z) \) are polynomials, admits no transcendental entire solutions of finite order. 

However, equation \eqref{1.2} has   solutions when  $q$
 is a transcendental function.  For example \cite{Wen}, \( f_1(z) = e^z + 1 \) and \( f_2(z) = e^z - 1 \) both solve
$f(z)^2 - 2e^z f(z - \log 2) = 1$.  Wen-Heittokangas-Laine \cite[Theorem 1.1(a)]{Wen} found  that every entire and finite-order solution $f$ of
\begin{align}
\label{1.3}
f^n(z)+q(z)e^{Q(z)}f(z+c)=P(z) 
\end{align}
satisfies $\rho(f)=\deg (Q)$ and is of mean type, where $ q(z) $, $ Q(z), P(z) $ are polynomials,  $ n \geq 2 $ are integers, and $ c \in \mathbb{C} \setminus \{0\} $.
 Hence it seems plausible that such \(f\) could be an exponential polynomial of the form
\begin{equation}
\label{1.4}
f(z) = P_1(z)e^{Q_1(z)}+\cdots + P_l(z)e^{Q_l(z)},
\end{equation}
where \(P_j\)'s and \(Q_j\)'s are polynomials in \(z\).  
 Motivated by this observation, Wen-Heittokangas-Laine  \cite[Theorem 1.1]{Wen}  classified
the exponential polynomial solutions of \eqref{1.3}.

\begin{thmx}{\cite[Theorem 1.1]{Wen}}
\label{thB}
 Let \(n \geq 2\) be an integer, let \(c \in \mathbb{C} \setminus \{0\}\), and let \(q(z), Q(z), P(z)\) be polynomials such that \(Q(z)\) is not a constant and \(q(z) \not\equiv 0\). Then we identify the finite order entire solutions \(f\) of equation \eqref{1.3} as follows:
\begin{enumerate}
    \item Every solution \(f\) satisfies \(\rho(f)=\deg (Q)\) and is of mean type.
    \item Every solution \(f\) satisfies \(\lambda(f)=\rho(f)\) if and only if \(P(z)\not\equiv 0\).
    \item A solution \(f\) belongs to \(\Gamma_0\) if and only if \(P(z)\equiv 0\). In particular, this is the case if \(n \geq 3\).
    \item If a solution \(f\) belongs to \(\Gamma_0\) and if \(g\) is any other finite order entire solution to \eqref{1.3}, then \(f = \eta g\), where \(\eta^{n - 1}=1\).
    \item If \(f\) is an exponential polynomial solution of the form \eqref{1.4}, then \(f\in\Gamma_1\). Moreover, if \(f\in\Gamma_1\setminus\Gamma_0\), then \(\rho(f)=1\).
\end{enumerate}
Here $\Gamma_1 = \{h = \mathrm{e}^{\alpha(z)} + d: d\in\mathbb{C} \text{ and } \alpha \text{ polynomial}, \alpha\neq \text{const.}\},
\Gamma_0 = \{h = \mathrm{e}^{\alpha(z)}: \alpha \text{ polynomial}, \alpha\neq \text{const.}\}.$
\end{thmx}

All the meromorphic solutions of equation \eqref{1.3} must be entire, which was proved by Liu \cite[Remark 1.(a)]{liuk} using an idea that originated from Naftalevich \cite{N 1976}. 
Liu \cite{liuk}  proceeds to classify the exponential polynomial solutions of 
\begin{equation}
\label{1.5}
f(z)^n + q(z)e^{Q(z)}f^{(k)}(z + c)=P(z), 
\end{equation}
where \(q(z)\) is entire, \(Q(z), P(z)\) are polynomials, \(k\) and \(n \geq 2\) are integers, and \(c\in\mathbb{C}\setminus\{0\}\). 

\begin{thmx}{\cite[Theorem 1.1]{liuk}}
\label{thC}
 Let \(q(z), P(z), Q(z)\) be polynomials such that \(Q(z)\) is not a constant, \(q(z)\not\equiv 0\), \(k\geq 1\) and \(n\geq 2\). Then, the finite order transcendental entire solution \(f\) of \eqref{1.5} should satisfy:

\begin{enumerate}
    \item Every solution \(f\) satisfies \(\rho(f)=\deg(Q(z))\) and is of normal type.
    \item Every solution \(f\) satisfies \(\lambda(f)=\rho(f)\) if and only if \(P(z)\not\equiv 0\).
    \item A solution \(f\) belongs to \(\Gamma_0'\) if and only if \(P(z)\equiv 0\). In particular, this is the case of \(n\geq 3\).
    \item If the solutions \(f, g\in\Gamma_0\), then \(f = \eta g\), where \(\eta^{n - 1}=1\).
    \item If \(f(z)\) is an exponential polynomial solution of the form \eqref{1.4}, then \(f\in\Gamma_1'\).
\end{enumerate}
Here, $\Gamma_1'= \{p(z)e^{\alpha(z)} + h(z)\}$,
$\Gamma_0' = \{p(z)e^{\alpha(z)}\},$ where $h$, $p$ and $\alpha$ are polynomial such that $\alpha$  is nonconstant.
\end{thmx}

The role of exponential polynomials in the theories of
linear/non-linear differential equations, oscillation theory and differential-difference equations  was 
discussed by Heittokangas-Ishizaki-Tohge-Wen \cite{H2023}. Thirteen open problems are given in \cite{H2023}, one of which is related to \eqref{1.5} can be stated as follows.

\vspace{10pt} 

\textbf{ Problem 12 in \cite{H2023}. } Suppose that \(f\) is an entire solution of \eqref{1.5}, where \(q(z)\) is a polynomial. If \(f\in\Gamma_1'\setminus\Gamma_0'\), then is it true that \(\rho(f)=1\)?

\vspace{10pt} 

First, we provide all the finite order entire solutions of \eqref{1.5}. This complements Theorem \ref{thB} and Theorem \ref{thC}.

\begin{thmx}
\label{th 1.4}
 Let \(q(z), P(z), Q(z)\) be polynomials such that \(Q(z)\) is not a constant, \(q(z)\not\equiv 0\), \(k\geq 0\) and \(n\geq 2\). If the  $f$ is a finite order transcendental entire solution   of \eqref{1.5}, then  one of the following conclusions holds. 
\begin{enumerate}
    \item $P\equiv0$,  $f=He^{\frac{Q+Q_1}{n-1}}$, where  $Q_1$ and $H$ are polynomials such that $\deg (Q_1)=\deg (Q) -1$.
    \item $P\not\equiv0$, $n=2$,  
$f= CP^{\frac{1}{2}} - P^{\frac{1}{2}} \frac{1}{2} \int F_2e^{Q} P^{-\frac{1}{2}} \, dz $, where \(C\) is an arbitrary constant, $F_2$ is a small function of $f$. In particular,  If \(f(z)\) is an exponential polynomial solution of the form \eqref{1.4}, then $k=0$, $\deg (Q)=1$,  $q$ and $P$ are constants  and 
\begin{align*}
    f=\frac{-q}{2}e^Q+h, 
\end{align*}
where $h^2=P$ is a constant. 
\end{enumerate}
\end{thmx}

The following   examples   illustrate conclusion (1) and conclusion (2) of Theorem \ref{th 1.4}.

\begin{exa} $f=e^{z^2}$ 
solves the equation
\[
f(z)^2 - e^{z^2-2z-1}f(z+1) = 0.
\]
Here $n=2$, $P=0$, $Q=z^2-2z-1$, $Q_1=2z+1$ and $H=1$. 
\end{exa}

\begin{exa}
\cite[Remark 5]{liuk}
$f(z) = (z+1)e^z$ solves the equation
\[
f(z)^n - e(z+1)^{n-1}e^{(n-1)z}f'(z-1) = 0.
\]
Here $n\ge 2$, $P=0$, $Q=(n-1)z$, $Q_1=0$ and $H=z+1$. 
\end{exa}
\begin{exa}\cite[Equation 1.5]{Wen}
\( f_1(z) = e^z + 1 \) and \( f_2(z) = e^z - 1 \) both solve
\[
f(z)^2 - 2e^z f(z - \log 2) = 1.
\]
\end{exa}

Next,  we derive the following corollary by invoking Theorem \ref{th 1.4}, which can be used to solve Problem 12 in \cite{H2023}.

\begin{cor}
Under the assumptions of Theorem \ref{th 1.4}, if \(f\in\Gamma_1'\setminus\Gamma_0'\), then \(\deg Q = 1\) and  \(f = \frac{-q}{2} e^Q + h\), where \(H\) is a polynomial and \(h\) is a constant. 
\end{cor}

\section{Proof of Theorem \ref{th 1.4}}

\begin{proof}
Let $f^{(k)}(z+c)=f^{(k)}_c$. It is easy to see $f$ is transcendental. Otherwise,  we can get $qf^{(k)}_c\equiv 0$, which is impossible.   If $P\equiv0$, then \eqref{1.5} reduces into 
 \begin{align}
    \label{2.1}
    f^n+qe^Qf^{(k)}_c=0.
 \end{align}
From \eqref{2.1}, the zero of $f$ must be the  zero of $qf^{(k)}_c$. Thus $F=\frac{-qf^{(k)}_c}{f}$ is an entire function. By the difference logarithmic derivative lemma \cite[Theorem 2.1]{halburd2006}, we get $T(r,F)=m(r,F)=S(r,f)$. Rewriting \eqref{2.1} into $f^{n-1}=Fe^Q$. By applying Theorem \ref{thC}-(1) and (3), we can obtain $f = p(z)e^{\alpha(z)}$  and $\deg (Q)=\deg (\alpha)$, where $p$ and $\alpha$ are polynomial such that $\alpha$  is nonconstant. Thus  $F=\frac{-qf^{(k)}_c}{f}=p_1e^{\alpha(z+c)-\alpha}$, where $p_1$ is a polynomial. We deduce $f^{n-1}=Fe^Q=p_1e^{Q+Q_1}$, where $Q_1=\alpha(z+c)-\alpha$ and $\deg (Q_1)=\deg (Q) -1$. Now we can get $f=He^{\frac{Q+Q_1}{n-1}}$, where $H$ is a polynomial such that $H^{n-1}=p_1$. 

 Next we consider the case $P\not\equiv0$. Differentiating equation \eqref{1.5} yields
\begin{align}
\label{2.2}
nf^{n-1}f'+q'e^Qf^{(k)}_c+qe^QQ'f^{(k)}_c+qe^Qf^{(k+1)}_c=P'.
\end{align}
Multiply equation \eqref{1.5} by $\frac{P'}{P}$  and subtract equation \eqref{2.2} yields 
\begin{align}
\label{2.3}
    \frac{P'}{P}f^n-nf^{n-1}f'=F_1e^Q, 
\end{align}
where $F_1=q'f^{(k)}_c+qQ'f^{(k)}_c+qf^{(k+1)}_c-\frac{P'}{P}qf^{(k)}_c$.  If $F_1\equiv0$, then \eqref{2.3} yields  $\frac{P'}{P}f^n-nf^{n-1}f'=0$. 
Through integration, we can obtain $f^n=c_1P$, where $c_1$ is a nonzero constant. This contradicts the fact that $f$ is a transcendental function. Therefore $F_1\not\equiv0$.

From equation \eqref{2.3}, we can deduce that if a zero of $f$ is not a pole of $\frac{P'}{P}$, then it must be a zero of $F_1$. Let $F_2=\frac{F_1}{f}$, 
since $P$ is a polynomial,  we can get $N(r,\frac{F_1}{f})=S(r,f)$.  By the difference logarithmic derivative lemma \cite[Theorem 2.1]{halburd2006}, we get $m(r,F_2)=S(r,f)$. Thus $T(r,F_2)=S(r,f)$, we rewrite \eqref{2.3} into 
\begin{align}
\label{2.4}
\frac{P'}{P}f^{n-1}-nf^{n-2}f'=F_2e^Q.
\end{align}
Next we consider two cases $n\ge3$ and $n=2$. 

{\textbf{Case 1.}} We  consider  the case $n\ge3$ at first. Since $n\ge 3$, we  can deduce that if a zero of $f$ is not a pole of $\frac{P'}{P}$, then it must be a zero of $F_2$.   Thus $N(r,\frac{1}{f})\le N(r,\frac{1}{F_2})+S(r,f)=S(r,f)$. By applying Hadamard's factorization theorem for finite order entire functions, we can obtain $f=He^g$, where $H$ is a small function of $f$ and $g$ is a polynomial.  Let $H_1=\frac{f^{(k)}_c}{f}$, it is easy to see $N(r,H_1)=S(r,f)$.  By the difference logarithmic derivative lemma \cite[Theorem 2.1]{halburd2006}, we get $m(r,H_1)=S(r,f)$. Thus $H_1$ is a small function of $f$. We may rewrite $f^{(k)}_c=\frac{f^{(k)}_c}{f}f=H_1He^g$. Now substituting the expressions of $f$ and   $f^{(k)}_c$ into \eqref{1.5}, we can get  
\begin{align}
\label{2.5}
H^ne^{ng}+qe^QH_1He^g=P. 
\end{align}
By \eqref{2.5}, we can  apply  the second fundamental theorem concerning three small functions \cite[Theorem 2.5]{hayman} to $H^ne^{ng}$, then 
\begin{align*}
T(r,H^ne^{ng})&\le N(r,\frac{1}{H^ne^{ng}}) +N(r,H^ne^{ng})+N(r,\frac{1}{H^ne^{ng}-P})+S(r,e^{ng})\\
&\le S(r,f)+S(r,e^{ng}),
\end{align*}
which  contradicts the
fact that $T(r,f)=T(r,e^g)+S(r,f)$.  Thus $n\ge 3$ is impossible. 

{\textbf{Case 2.}} Next, we  consider  the case $n=2$. From \eqref{2.4},  we deduce 
\begin{align}
\label{2.6}
\frac{P'}{P}f-2f'=F_2e^Q. 
\end{align}
Thus, the general solution is
\[
f= CP^{\frac{1}{2}} - P^{\frac{1}{2}} \frac{1}{2} \int F_2e^{Q} P^{-\frac{1}{2}} \, dz 
\]
where \(C\) is an arbitrary constant.

Let   \(f(z)\) be  an exponential polynomial solution of the form \eqref{1.4}. By applying Theorem \ref{thC}-(2) and (5), we deduce that $f(z)=p(z)e^{\alpha(z)} + h(z)$, where $p(z), h(z)$ are polynomials such that $h\not\equiv0$ and $\alpha(z) $ is a nonconstant polynomial. Substituting $f(z)=p(z)e^{\alpha(z)} + h(z)$ into  \eqref{2.6}, we get 
\begin{align}
\label{2.7}
(\frac{P'}{P}p-2p'-2p\alpha')e^\alpha+\frac{P'}{P}h-2h'=F_2e^Q. 
\end{align}
We claim $A=\frac{P'}{P}p-2p'-2p\alpha'\not\equiv0$. Otherwise, by integration, we can obtain that $e^{2\alpha}=c_1\frac{P}{p^2}$, where $c_1$ is a constant. This contradicts the
fact that $P$  and $p$ are polynomials.  Therefore $T(r,f)=T(r,e^\alpha)+O(\log r)=T(r, Ae^\alpha)+O(\log r)$. Since $F_2$ is a small function of $f$, then $F_2$ is also a small function of $Ae^\alpha$.
 
We claim $\frac{P'}{P}h-2h' \equiv0$. Otherwise, $\frac{P'}{P}h-2h'\not\equiv0$, by \eqref{2.7}, we can  apply  the second fundamental theorem concerning three small functions \cite[Theorem 2.5]{hayman} to $Ae^\alpha$, then 
\begin{align*}
T(r,Ae^\alpha)&\le N(r,\frac{1}{Ae^\alpha}) 
+N(r,Ae^\alpha) \nonumber
+N(r,\frac{1}{Ae^\alpha+\frac{P'}{P}h-2h'})+S(r,f)\nonumber \\
&\le N(r,\frac{1}{F_2}) +S(r,f ) \nonumber \\
&\le S(r,f ). 
\end{align*}
This contradicts the fact that   $T(r,f) =T(r, Ae^\alpha)+O(\log r)$. Thus  $\frac{P'}{P}h-2h'\equiv0$, by integration, we can obtain $h=c_2P^{\frac{1}{2}}$, where $c_2$ is a nonzero constant. From \eqref{2.7}, we get $e^\alpha=\frac{F_2}{\frac{P'}{P}p-2p'-2p\alpha'}e^Q$. Thus $T(r,e^Q)=T(r,e^\alpha)+S(r,f)$.
 Now we rewrite $f(z)=p(z)e^{\alpha(z)} + h(z)$ into $f(z)=p_1(z)e^Q+h(z)$, where $p_1=p\frac{F_2}{\frac{P'}{P}p-2p'-2p\alpha'}$ is a small entire function of $e^Q$ and $h=c_2P^{\frac{1}{2}}$. Thus $f^{(k)}_c=v_{k}e^{Q(z+c)-Q(z)}e^{Q(z)}+m_{k}$, where $v_k$ is  a  differential difference polynomial  in $p_1$ and $Q$, and  $m_{k}=h^{(k)}(z+c)$.

 Substituting the expressions of $f$ and  $f^{(k)}_c$ into \eqref{1.5}, we get 
 \begin{align}
 \label{2.9}
(p_1^2+qv_ke^{Q(z+c)-Q(z)})e^{2Q}+(2p_1h+qm_{k})e^Q+h^2=P. 
 \end{align}
Using Borel's Lemma \cite[p. 77]{yang2003} to \eqref{2.9} yields $h^2=P$, $p_1^2+qv_ke^{Q(z+c)-Q(z)}\equiv0$ and  $2p_1h+qm_{k}\equiv0$.  
 $2p_1h+qm_k\equiv0$ means $p_1=\frac{-qm_k}{2h}$ is a polynomial. Thus $v_k$ is  also a polynomial. 
 $p_1^2+qv_ke^{Q(z+c)-Q(z)}\equiv0$  yields $\deg (Q)=1$. Let $Q(z)=\beta z+\gamma$, where $\beta$ is a nonzero constant and $\gamma$ is a constant. Then $v_k=
  \sum_{j=0}^{k} \binom{k}{j} p_1^{(j)}(z+c) \beta^{k-j}$, $\deg v_k =\deg p_1$.   By  $p_1^2+qv_ke^{Q(z+c)-Q(z)}\equiv0$, we obtain  $2\deg (p_1)=\deg (q)+\deg p_1$. Thus $\deg (p_1)=\deg (q)$. However $2p_1h+qm_{k}\equiv0$ implies $\deg (h)+\deg (p_1)=
  \deg (q)+\deg (h)-k$.  Therefore $k=0$, $v_k=p_1(z+c)$. By applying \cite[Lemma 2.4]{Wen} to the equation \( p_1^2 + q p_1(z + c) e^{Q(z + c) - Q(z)} \equiv 0 \), we can conclude that \( p_1 \) and $q$ are constants.  From $2p_1h+qh(z+c)\equiv0$, we can deduce that  $h$ is a constant. 
Now  we draw a conclusion  that $n=2$, $k=0$,   and 
\begin{align*}
    f(z)=\frac{-q}{2}e^Q+h, 
\end{align*}
where $h^2=P$ is a constant  and $\deg (Q)=1$. 
\end{proof}

	\section*{Declarations}
	\begin{itemize}
		
		\item\noindent{\bf Funding}
  Jianren Long was supported by the National Natural Science Foundation of China (Grant No. 12261023, 11861023) and Xuxu Xiang was supported by Graduate Research Fund Project of Guizhou Province (2025YJSKYJJ107)
		\item \noindent{\bf Conflicts of Interest}
		The authors declare that there are no conflicts of interest regarding the publication of this paper.

        \item\noindent{\bf Author Contributions}
All authors contributed to the study conception and design. All authors read and approved the final manuscript.

%\item\noindent{\bf Acknowledgements} The authors would like to thank Prof. Zhitao Wen of the Shantou University  for giving a lot of valuable suggestions during the preparation of this paper.
	\end{itemize}

\end{document}